%% file: gluingpaper_lms_revised.tex
\documentclass[11pt]{article}
\usepackage[a4paper, hmarginratio=1:1]{geometry} 
\usepackage{graphicx}
\usepackage{amssymb}
\usepackage{epstopdf}
\usepackage{mathtools}
\usepackage{tikz}
\usepackage{tikz-cd}
\usepackage{tkz-euclide}
\usepackage{subdepth} 
\usepackage{titlefoot}
\usetkzobj{all}

\usepackage{comment}
\usepackage{subfig}
\usetikzlibrary{arrows,chains,matrix,positioning,scopes}
\usetikzlibrary{decorations.markings}

\usepackage[font=small, labelfont=bf, width=.8\textwidth]{caption}

\usepackage{amsmath}
\usepackage{amsthm}
\usepackage{cancel}

\newtheorem{theorem}{Theorem}[section]

\newtheorem{proposition}[theorem]{Proposition}

\newtheorem{lemma}[theorem]{Lemma}

\theoremstyle{definition}
\newtheorem{definition}[theorem]{Definition}

\newtheorem{remark}[theorem]{Remark}

\usepackage[scaled]{helvet} % ss
\usepackage{courier} % tt
\usepackage{fourier} % math
\normalfont
\usepackage[T1]{fontenc}
\let\phi\varphi

\usepackage{enumitem}
\setdescription{leftmargin=\parindent,labelindent=\parindent}

\usepackage[numbers]{natbib}

\newcommand{\gdots}{, \, \dots, \,}

\newcommand{\cat}[1]{${\rm CAT}(#1)$}
\newcommand{\ice}{\textnormal{ICE }}

\usepackage[colorlinks=true, linkcolor=green!50!black, 
         citecolor=green!50!black, urlcolor=green!50!black]{hyperref}
\title{A gluing theorem for negatively curved complexes}
\author{Samuel Brown\thanks{Department of Mathematics, University College London, Gower Street, London, WC1E 6BT,  s.brown.12@ucl.ac.uk.}\thanks{MSC classes: 20F65, 20F67.} \thanks{This work was supported by the EPSRC.}}

\begin{document}

\maketitle %\unmarkedfntext{MSC classes: 20F65, 20F67}

\begin{abstract}
A simplicial complex is called \emph{negatively curved} if all its simplices are isometric to simplices in hyperbolic space, and it satisfies Gromov's Link Condition. We prove that, subject to certain conditions, a compact graph of spaces whose vertex spaces are negatively curved 2-complexes, and whose edge spaces are points or circles, is negatively curved. As a consequence, we deduce that certain groups are \cat{-1}. These include hyperbolic limit groups, and hyperbolic groups whose JSJ components are fundamental groups of negatively curved 2-complexes---for example, finite graphs of free groups with cyclic edge groups. 
\end{abstract}

\section{Introduction}

A common theme in geometric group theory is to understand the consequences of negative and non-positive curvature. For finitely generated infinite groups, the most widely studied definition of negative curvature is the notion of $\delta$-hyperbolicity, as originally studied by Gromov \cite{gromov87}. A group is called $\delta$-hyperbolic if it acts geometrically on a $\delta$-hyperbolic metric space---in which negative curvature is measured using the $\delta$-thinness condition on geodesic triangles---and a group is called hyperbolic if it is $\delta$-hyperbolic for some $\delta$. A stronger notion of negative curvature, also measured by geodesic triangles, is the \cat{-1} condition, which requires that all geodesic triangles be thinner than the corresponding triangles in the hyperbolic plane. Similarly, the \cat{0} condition requires that triangles are thinner than corresponding Euclidean triangles; a space which is \cat{-1} is also \cat{0}.

\begin{sloppypar}
For metric spaces, the interplay between these definitions is well understood: every \mbox{\cat{-1}} space is hyperbolic, while (proper, cocompact) \cat{0} spaces are hyperbolic whenever they do not contain an isometrically embedded Euclidean plane \cite{bh}. The story for groups is much less clear; while \cat{-1} groups (that is, groups possessing a geometric action on a \cat{-1} space) are all hyperbolic, it is an open question whether all hyperbolic groups are \cat{-1}, or even \cat{0}. The difficulty arises because the \cat{k} condition requires negative curvature at a small scale, whereas hyperbolicity is only sensitive to the geometry at larger scales. A strong form of this question appears as Question 1.5 in Bestvina's famous list \cite{bestvinalist}, credited to Michael Davis; it asks whether the Rips complex $P_R(G)$ of a hyperbolic group $G$ possesses a negatively curved metric for sufficiently large $R$. A positive answer would not only show that all hyperbolic groups were \cat{-1}, but would give an explicit description of the \cat{-1} space and the action.

The main result of interest in this paper is to settle the question for the case of \emph{limit groups}---a class of groups introduced by Sela \cite{sela2001} and widely studied due to their usefulness in understanding homomorphisms from a finitely generated group to a free group (for more details, see section \ref{sec:lg} and the references there). Limit groups were shown in \cite{alibegovicbestvina06} to be \cat{0}, and we improve this to the following:

\newtheorem*{thminta}{Theorem \ref{thm:limcat-1}}
\begin{thminta}
	Let $G$ be a limit group. Then $G$ is \cat{-1} if and only if $G$ is hyperbolic.
\end{thminta}

\end{sloppypar}
By a \emph{negatively curved simplicial complex}, we mean a simplicial complex whose simplices are modelled on simplices in (possibly rescaled) hyperbolic space. In order to prove Theorem \ref{thm:limcat-1}, we devise a gluing theorem (Theorem \ref{thm:main}) for two-dimensional negatively curved simplicial complexes, in the spirit of the gluing theorems presented in \cite[II.11]{bh}. Essentially, we would like to take two negatively curved 2-complexes, glue them together along a tube, and find a hyperbolic metric on the resulting complex. Imposing a hyperbolic metric on the tube is problematic, since hyperbolic annuli have a non-geodesic boundary component. To get around this issue, we modify the metric on the two complexes, concentrating negative curvature at vertices, and giving ``room'' to glue in the tube. Care is needed to ensure that the pieces we glue on do not themselves combine to give positive curvature, and this is dealt with in Lemma \ref{lemma:transverse}. 

The most closely related theorem to ours in the literature is Bestvina and Feighn's gluing theorem \cite{bestvinafeighn92} for $\delta$-hyperbolic spaces, and thus for hyperbolic groups. Our theorem is complementary to theirs; our hypotheses are stronger, but so is our conclusion. 

We present the proof of the gluing theorem in section \ref{sec:mainthm}, and then in section \ref{sec:lg} we introduce limit groups. After giving some background and listing some useful properties, we will exploit the rich structure theory of limit groups to allow us to apply our gluing theorem, and hence to prove Theorem \ref{thm:limcat-1}. In section \ref{sec:apps} we provide two more applications of the gluing theorem, showing that it can be applied to the JSJ decomposition of a torsion-free hyperbolic group (Theorem \ref{prop:jsjnc}), and deducing that graphs of free groups with cyclic edge groups are \cat{-1} (Theorem \ref{thm:gofg}). A consequence of the application to the JSJ decomposition, together with the Strong Accessibility Theorem of Louder--Touikan \cite{loudertouikan13} (see also \cite{delzantpotyagailo01}), is that we may reduce the question of whether a hyperbolic group is 2-dimensionally \cat{-1} to its rigid subgroups (those which do not admit a non-trivial free or cyclic splitting).

%\newtheorem*{thmintb}{Theorem \ref{thm:open}}
%\begin{thmintb}
%	All torsion-free, geometrically 2-dimensional hyperbolic groups are \cat{-1} if and only if all rigid, torsion-free, geometrically 2-dimensional hyperbolic groups are \cat{-1}.
%\end{thmintb}

Our gluing theorem will be expressed using the language of \emph{graphs of spaces}, as formalised by \cite{scottwall79}, and this is where we begin.
\section{Graphs of spaces and CAT(\emph{k}) geometry}

\subsection{Graphs of spaces}\label{ss:gos}

Recall that a \emph{graph} $\Gamma$, as defined by Serre \cite{trees}, consists of a set $V=V(\Gamma)$ of \emph{vertices}, a set $E=E(\Gamma)$ of \emph{edges}, and maps $\iota \colon E \rightarrow V$, $\tau \colon E \rightarrow V$, $\bar{\,} \colon E \rightarrow E$ satisfying $\bar{e} \neq e$, $\bar{\bar{e}}=e$ and $\tau(e)=\iota(\bar{e})$. The corresponding topological graph, also denoted by $\Gamma$, has vertices $V$ and one edge $e=\left\{\iota(e), \tau(e) \right\}$ for each edge pair $(e, \bar{e})$.

Throughout, if not explicitly specified, we will assume all our spaces are CW complexes (they will usually be explicit polyhedral complexes). 

\begin{definition}
	A \emph{graph of spaces} $\left( X, \Gamma_X \right)$ consists of the following:
	\begin{enumerate}
		\item A graph $\Gamma_X$, called the \emph{underlying graph}.
		\item For each vertex $v$ of $\Gamma_X$, a connected \emph{vertex space} $X_v$.
		\item For each edge pair $(e, \bar{e})$ of $\Gamma_X$, a connected \emph{edge space} $X_e=X_{\bar{e}}$ and a pair $\left( \partial_e, \partial_{\bar{e}} \right)$ of $\pi_1$-injective \emph{attaching maps} from $X_e$ to $X_{\iota(e)}$, $X_{\tau(e)}\left(=X_{\iota(\bar{e})}\right)$ respectively.
	\end{enumerate}
\end{definition}

Given the above data, we associate a space $X$, called the \emph{total space}, as follows. Take a copy of $X_v$ for each $v$, and a copy of $X_e \times [0,1]$ for each edge pair $\left(e,\bar{e}\right)$. Now glue $X_e \times \{0\}$ to $X_{\iota(e)}$ using $\partial_e$, and glue $X_e \times \{1\}$ to $X_{\tau(e)}$ using $\partial_{\bar{e}}$. The edge space $X_e$ is embedded inside $X$ as $X_e \times \{\frac{1}{2}\}$. We will occasionally refer to $X_e \times [0,1]$ as an \emph{edge cylinder}.

We will say that the total space $X$ has a \emph{graph of spaces decomposition}, or simply is a \emph{graph of spaces}. The ambiguity is the potential existence of two different graph of spaces decompositions of a given topological space, but that will not concern us here. 

To each graph of spaces, we may naturally associate a graph of groups with the same underlying graph, by replacing vertex and edge spaces with their fundamental groups and the attaching maps with the induced maps on fundamental groups (after choosing basepoints). The fundamental group of the graph of groups (as defined in \cite{trees}) is then isomorphic to the fundamental group of the total space. When we discuss \emph{vertex groups} and \emph{edge groups} of a graph of spaces, we are referring to the vertex and edge groups of this corresponding graph of groups. A \emph{splitting} is a synonym for ``graph of groups decomposition''.

Given a graph of spaces $X$, there is a natural projection $\phi_X \colon X \rightarrow \Gamma_X$ which maps each vertex space $X_v$ to the vertex $v$, and each copy of $X_e \times [0,1]$ to the edge $e$ by projection onto the second factor. A map $f \colon X \rightarrow X'$ of graphs of spaces is one that respects the graph of spaces decomposition, in the sense that there is a graph homomorphism $\gamma \colon \Gamma_X \rightarrow \Gamma_{X'}$ such that $\gamma \circ \phi_X = \phi_{X'} \circ f$. %Note that we may define such a map on vertex and edge spaces only, provided that our definitions are compatible with attaching maps. More specifically, writing $v'$ for $\gamma(v) \in V \left( \Gamma_{X'} \right)$ and $\partial'$ for the attaching maps in $X'$, a family of maps $\set{f_v \colon X_v \rightarrow X'_{v'}}{v \in V \left( \Gamma_V \right) } \cup \set{f_e \colon X_e \rightarrow X'_{e'}}{e \in E \left( \Gamma_V \right) }$ defines a map of graph of spaces $f \colon X \rightarrow X'$ if and only if $f_v \circ \partial_e = \partial'_{e'} \circ f_e$ whenever $\iota(e)=v$.

%A \emph{sub-graph of spaces} of $X$ is a graph of spaces $X'$ equipped with an injective map of graphs of spaces to $X$. A \emph{core} of a graph of spaces $X$ is a sub-graph of spaces $X'$, with finite underlying graph, such that the inclusion $X' \hookrightarrow X$ is a $\pi_1$-isomorphism. For example, for any graph of spaces $X$ with finitely generated fundamental group, the underlying graph $\Gamma_X$ has finite rank, and hence there exists a core: namely the sub-graph of spaces induced by the core graph for $\Gamma_X$\footnote{That is, a subgraph of $\Gamma_X$ consisting of a representative loop for each generator of $\pi_1 \left( \Gamma_X \right)$.}.

If $X$ is a graph of spaces and $\hat{X} \rightarrow X$ is a covering map, then $\hat{X}$ inherits a graph of spaces structure. The vertex and edge spaces are preimages of vertex and edge spaces of $X$, and the restriction of the covering map to a vertex (or edge) space of $\hat{X}$ is a covering map to the corresponding vertex (or edge) space of $X$.

There may be many graphs of spaces associated to the same graph of groups, but fortunately we have strong control over their homotopy type, thanks to the following lemma (a consequence of Whitehead's theorem, and the fact that a graph of aspherical spaces is aspherical):

\begin{lemma}[\cite{scottwall79}]\label{lemma:goshe}
	Let $X$ and $Y$ be two graphs of spaces associated to the same graph of groups. Suppose all the vertex and edge spaces of $X$ and $Y$ are aspherical. Then $X$ and $Y$ are homotopy equivalent. 
\end{lemma}

In some accounts of the theory of graphs of groups and graphs of spaces (for example \cite{hatcherat}), a graph of groups is chosen first and then a graph of spaces is constructed by choosing a $K \left(G_v, 1\right)$ for each vertex group $G_v$, a $K \left(G_e, 1\right)$ for each edge group $G_e$, and then realising the attaching maps by $\pi_1$-injective maps between these spaces. The above lemma then essentially says that this construction is well-defined up to homotopy. In particular, we can safely replace vertex and edge spaces of a graph of spaces with homotopy equivalent spaces, and attaching maps with freely homotopic maps, without changing the homotopy type of the graph of spaces. We will make use of this fact regularly. 

\subsection{Metrizing graphs of spaces}\label{ss:mgos}

A priori, a graph of spaces $X$ is not equipped with a metric. However, if we have a metric on each vertex and edge space, then we may metrize the cylinders $X_e \times [0,1]$ using the product metric (with the standard metric on $[0,1]$), and then the quotient pseudometric on $X$ will be a true metric if the attaching maps are suitably nice. A sufficient condition is that they be \emph{local isometries}, in the sense that each point in an edge space $X_e$ has a neighbourhood on which $\partial_e$ restricts to an isometric embedding.

Even in the case that the attaching maps are not local isometries---or even when metrics on the edge spaces are not specified---a graph of spaces may possess other metrics. Indeed, Theorem \ref{thm:main} defines a metric on a graph of spaces which does not come from a product metric on the edge space cylinders.

%We will generally consider graphs of spaces in which vertex and edge spaces are geometric simplicial complexes. In this setting, we will use the following characterisation of local isometries:

%\begin{proposition}
%	Let $f \colon A \rightarrow B$ be a combinatorial map of metric simplicial complexes. Then $f$ is a local isometry if and only if its restriction to each edge $e$ of $A$ is an isometry to the edge $f(e)$, and for every vertex $v$ of $A$, the induced map $f \colon \mathrm{link}\left(v,A\right) \rightarrow \mathrm{link}\left(f(v),B\right)$
%\end{proposition}

\subsection{Malnormality}\label{ss:malfam}

The following is an algebraic criterion for families of subgroups of a given group, and it will be a crucial hypothesis for our gluing theorem.

\begin{definition}
	A subgroup $H$ of a group $G$ is called \emph{malnormal} if $x^{-1}Hx \cap H=\{1\}$ whenever $x \notin H$. A family of subgroups $H_1, \dotsc, H_n$ of $G$ is a \emph{malnormal family} if each $H_i$ is malnormal, and in addition $x^{-1}H_{i}x \cap y^{-1}H_{j}y=\left\{1\right\}$ for any $i \neq j$ and any $x$, $y \in G$.
\end{definition}

When describing hierarchies of groups in various contexts, it is common to see the hypothesis that incident edge groups should form a malnormal (or \emph{almost malnormal}, where finite intersections are allowed) family of subgroups in each vertex group. This is seen, for example, in Wise's hierarchy for virtually special groups \cite{wisenotes, wisemanu}, and as a special case in \cite{hsuwise10}. It is stronger than the ``annuli flare'' condition used by Bestvina and Feighn in their gluing theorem for $\delta$-hyperbolic spaces \cite{bestvinafeighn92}. The geometric consequences of malnormality for a subgroup are not fully understood (in particular, the circumstances under which it implies quasiconvexity), and it may be possible to replace it with a weaker assumption in future---see Remark \ref{rmk:malfam} for more details.

%Say something about Bestvina-Feighn here

\begin{comment}
The following terminology is due to Wise \cite{wise00, wise02}:

\begin{definition}
	A graph of groups (or graph of spaces) is called \emph{thin} if, for each vertex group, the set of incident edge groups is a malnormal family.
\end{definition}

\begin{remark}
	Wise's original definition allows \emph{almost malnormal} familes, where intersections between conjugates are allowed to be finite. However, we will be concerned only with cyclic edge groups and therefore will not require this generalisation.
\end{remark}
\end{comment}

\subsection{CAT(\emph{k}) spaces and groups}\label{ss:catk}

The \cat{k} criterion gives a method for describing the curvature of a metric space by comparing triangles in the space to triangles in a space of fixed curvature. We will recall the basics for the convenience of the reader, but precise details can be found in \cite{bh}. 

Following \cite[Definition I.2.10]{bh}, for any $k<0$, we denote by $M^n_k$ the space obtained by multiplying the metric on $n$-dimensional hyperbolic space $\mathbb{H}^n$ by $1/\sqrt{-k}$. For $k=0$, $M^n_k$ denotes Euclidean space $\mathbb{E}^n$, and for $k<0$, $M^n_k$ denotes the rescaled sphere obtained by multiplying the metric on the $n$-dimensional unit sphere $\mathbb{S}^n$ by $1/\sqrt{k}$.

\begin{definition}
For $k \leq 0$, a geodesic metric space $Y$ is called \cat{k} if any geodesic triangle in $Y$ is \emph{thinner} (see Figure \ref{fig:catk}) than the triangle with the same side lengths in $M^2_k$. For $k > 0$, we require this only for triangles of perimeter less than $2\pi/\sqrt{k}$ (twice the diameter of the sphere $M^2_k$). 
\end{definition}

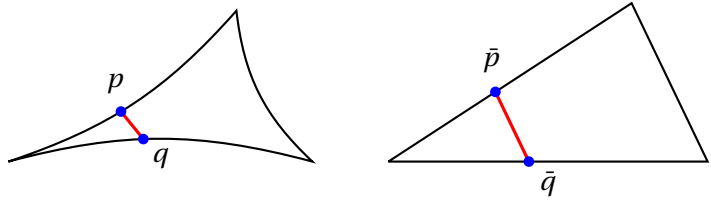
\begin{figure}[h]
	\centering
	\begin{tikzpicture}
	\draw[thick]
	(0,0) to[bend right=15] coordinate[pos=0.44] (A) node[above left] {$p$} (3,2)  to[bend right=20] (4,0) to[bend right=15] coordinate[pos=0.56] (B) node[below] {$q$}  (0,0) ; 
	\draw[very thick, red]
	(A) -- (B);
	\fill[blue] (A) circle (2pt);
	\fill[blue] (B) circle (2pt);
	\draw[thick]
	(5,0) to coordinate[pos=0.44] (C) node[above left] {$\bar{p}$} (8.2,2.1) -- (9.2,0) to coordinate[pos=0.56] (D) node[below] {$\bar{q}$} (5,0) ; 
	\draw[very thick, red]
	(C) -- (D);
	\fill[blue] (C) circle (2pt);
	\fill[blue] (D) circle (2pt);		
	\end{tikzpicture} 
	\caption{The left hand triangle is thinner than the right, as $d(p,q) \leq d(\bar{p},\bar{q})$ for any choice of points $p$ and $q$. Note that equality is permitted.}
	\label{fig:catk}
\end{figure}

\begin{definition}
	A space is said to be $\emph{locally}$ \cat{k} if every point has a neighbourhood which is \cat{k}.
\end{definition}

\begin{remark}\label{rmk:rescaling}
	If a space is (locally) \cat{k} for some $k<0$, then by multiplying the metric by $\sqrt{-k}$ we obtain a (locally) \cat{-1} space. That is, up to rescaling the metric (in particular, up to homeomorphism and homotopy equivalence), a space which is (locally) \cat{k} for some $k<0$ is (locally) \cat{k} for all $k<0$. 
\end{remark}

\begin{definition}
	An \emph{$M^n_k$-simplex} is the convex hull of $n+1$ points in general position in $M^n_k$. These $n+1$ points are the vertices of the simplex, and its faces are precisely the lower dimensional simplices obtained by taking the convex hull of a subset of the vertices.
\end{definition}

\begin{remark}\label{rmk:compsimplex}
	Let $k<0$ and let $S$ be an $M^2_k$-simplex with 1-skeleton $S^{(1)}$. Then for any $k'<0$, there exists a $M^2_{k'}$-simplex $S'$ with 1-skeleton $S'^{(1)}$ isometric to $S^{(1)}$. Each angle of $S'$ is strictly larger than the corresponding angle of $S$ if $k<k'<0$, and strictly smaller if $k'<k<0$. We say $S'$ is a \emph{comparison simplex for $S$ of curvature $k'$} (or a \emph{comparison $M^2_{k'}$-simplex} for $S$).
\end{remark}

%\begin{remark}\label{rmk:errorbrief}
%	In an earlier version of this paper, Remark \ref{rmk:compsimplex}, and subsequently Theorem \ref{thm:main}, were erroneously stated without the two-dimensional hypothesis. However, this hypothesis does not alter the main results concerning hyperbolic limit groups (Theorem \ref{thm:limcat-1}) or graphs of free groups with cyclic edge groups (Theorem \ref{thm:gofg}). See Remark \ref{rmk:error} for more details.
%\end{remark}

\begin{definition} \sloppy
	An \emph{$M_k$-simplicial complex} is a geometric simplicial complex, whose $n$-simplices are $M^n_k$-simplices glued by isometries of their faces.
\end{definition}
	%REALLY? SURE B-H CAN HAVE ANGLES >PI, WHY WOULDNT THEY
	An $M_k$-simplicial complex $K$ has a well-defined notion of angle, inherited from the isometries between each simplex and $M^n_k$. In particular, we can define angle inside each simplex using the standard angle in $M^n_k$. This induces a metric on the links of vertices inside simplices. When we glue together simplices, we glue by isometries of faces; the same holds for links, so we get a well defined path metric on the links of vertices in $K$. This in turn gives a notion of angle at a vertex between simplicial paths in $K$. The length of an edge of $\text{link}(v,K)$ is equal to the angle at $v$ in the corresponding 2-simplex of $K$. Note that the angle between two geodesics issuing from a vertex may be greater than $\pi$, and in the case where the link is disconnected, it may take the value $\infty$. For more details, see \cite{bh}.

\begin{remark}
	A simplicial path $\gamma$ in $K$ is locally geodesic if and only if, at each vertex $x$ along the path, the angle between the incoming and outgoing edges of $\gamma$ at $x$ is at least $\pi$. We call this the angle \emph{subtended} by $\gamma$ at $x$. Equivalently, $\gamma$ is locally geodesic if the distance in $\text{link}(v,K)$ between the two points corresponding to the incoming and outgoing edges of $\gamma$ is at least $\pi$. A closed simplicial path which is locally geodesic is called a \emph{closed geodesic}.
\end{remark}

\begin{theorem}[The Link Condition]
	Let $K$ be an $M_k$-simplicial complex. Then $K$ is locally \cat{k} if and only if for each vertex $v \in K$, $\text{link}(v,K)$ is \cat{1}.
\end{theorem}

\begin{remark}\label{rmk:2dcat-1}
	In the case where $K$ is 2-dimensional, the links of vertices are metric graphs, which are \cat{1} precisely when they do not contain any essential loops of length less than $2\pi$. This makes it particularly easy to check the link condition in the 2-dimensional case. 
\end{remark}

\begin{remark}\label{rmk:cat1}
	Given a \cat{1} space $X$, the space obtained by connecting two points $x$, $y \in X$ with an arc of length $l$ is \cat{1} if and only if $d(x,y)+l \geq 2\pi$. This is because no new non-degenerate triangles of perimeter $<2\pi$ are introduced. This is a useful fact when checking that links remain \cat{1}---and hence, spaces remain locally \cat{k} for some $k$---after gluing operations.
\end{remark}

For $k<0$, we will refer to an $M_{k}$-simplicial complex which is locally \cat{k} as a \emph{negatively curved simplicial complex (of curvature $k$)}, or just a \emph{negatively curved complex}. From now on, we will assume that all complexes are locally finite (see Remark \ref{rmk:deltafinite}), and will be concerned only with 2-dimensional complexes (see Remark \ref{rmk:error}).

\begin{definition}\label{def:compcomplex}
	If $K$ is a negatively curved simplicial 2-complex, then we may form a combinatorially isomorphic complex $K'$ by replacing each $M^2_k$-simplex $S$ with the comparison simplex $S'$ of curvature $k'$, where $k<k'<0$. Then $K'$ is an $M_{k'}$-simplicial 2-complex, which we call the \emph{comparison $M_{k'}$-complex} for $K$. Note that gluing maps between faces of simplices remain isometries, and so this is a well-defined $M_{k'}$-simplicial complex.
\end{definition}

There is a natural combinatorial isomorphism $' \colon K \rightarrow K'$, which can be realised by a homeomorphism. We also use the same notation for the induced combinatorial isomorphism $' \colon \text{link}(v,K) \rightarrow \text{link}(v',K')$ for each vertex $v \in K^{(0)}$.

\begin{lemma}
	Let $K$ and $K'$ be as in Definition \ref{def:compcomplex}. Then $K'$ is locally \cat{k'}; in particular, it is a negatively curved 2-complex.
\end{lemma}

\begin{proof}
By Remark \ref{rmk:compsimplex}, angles at vertices in $K'$ are larger than in $K$. Since links in $K$ and $K'$ are metric graphs, it follows that the map $' \colon \text{link}(v,K) \rightarrow \text{link}(v',K')$ strictly increases the distance between pairs of points. Remark \ref{rmk:2dcat-1} then implies that links in $K'$ are \cat{1}, and hence $K'$ is locally \cat{k'}. 
\end{proof}

We quantify this as follows.
\begin{definition}\label{def:delta}
	For a negatively curved simplicial 2-complex $K$ and comparison complex $K'$, the \emph{excess angle} $\delta$ is defined by:\[
	\delta(K',K)=\inf \left\{ \theta'-\theta \right\} \]
	where $\theta$ ranges over all vertex angles of 2-simplices $S \subset K$ and $\theta'$ is the corresponding angle in the comparison simplex $S' \subset K'$.  Equivalently: \[
	\delta(K',K)=\inf \left\{ d_{\text{link}(v',K')} ( a',b' ) - d_{\text{link}(v,K)} ( a,b ) \; \middle| \; a,\,b \in \text{link}(v,K)^{(0)}, \, v \in K^{(0)} \right\}
	\]
	
	If $K$ is finite (or more generally, if the set Shapes($K$) of isometry types of simplices is finite---see \cite{bh}), then $\delta(K',K)>0$. 
\end{definition}

\begin{remark}\label{rmk:angleexcess}
	Suppose $\gamma = \left( e_1 \gdots e_n \right)$ is a locally geodesic simplicial path in $K$, with $\iota(e_i)=v_{i-1}$, $\tau(e_i)=v_i$ for $i=1 \gdots n$. Then $\gamma' = \left( e'_1 \gdots e'_n\right)$ is a local geodesic in $K'$, and for each $i=1 \gdots n-1$ the angle subtended by $\gamma$ at $v'_i$ is at least $\pi+2\delta(K',K)$; that is: \[
		d_{\text{link}(v'_i,K')} \left( e'_i,e'_{i+1} \right) \geq \pi+2\delta(K',K).
			\] In the case where $\gamma$ is a closed local geodesic, the same also holds for the angle at $v'_0=v'_n$ between $e'_n$ and $e'_1$ and so $\gamma'$ is still a closed geodesic.
\end{remark}

\begin{remark}\label{rmk:deltafinite}
	We will typically use the excess angle only in the spirit of Remark \ref{rmk:angleexcess} above. In this setting, we can relax the requirement that $K$ or Shapes($K$) is finite, provided we insist that $K$ is \emph{locally} finite. This is because we only need to consider the finitely many angles around some finite subset of $K$, such a closed geodesic or family of closed geodesics. 
\end{remark}

In order to state the group theoretic consequences of our gluing theorem, we need some further definitions.

\begin{definition}
	We say a finitely generated group acts on a space \emph{geometrically} if it acts by isometries, properly discontinuously and cocompactly. 
\end{definition}

% %CAN WE DROP THE COCOMPACTNESS ASSUMPTION? REVISIT THIS SECTION

\begin{definition}
	A group is called \emph{\cat{k}} if it acts geometrically on a \cat{k} complex. A group is called \emph{freely \cat{k}} if it acts freely and geometrically on a \cat{k} complex. 
\end{definition}

\begin{definition}
	The \emph{geometric dimension} of a group $G$ is the minimum dimension of a $K(G,1)$. The \emph{ \cat{-1} dimension} of $G$ is the minimum dimension of a compact \cat{-1} complex which is a $K(G,1)$.
\end{definition}

\begin{remark} \label{rmk:cat-1dimcpct}
	The requirement that the $K(G,1)$ be compact in the definition of \cat{-1} dimension is to ensure that a group cannot have a defined \cat{-1} dimension unless it is \cat{-1}. Indeed, a version of Rips' construction can be used to build groups which have a non-compact \cat{-1} $K(G,1)$, but which are not finitely presented (see \cite[II.5]{bh}). In particular, they are not hyperbolic, and so cannot be \cat{-1}.
\end{remark}

\begin{remark}\label{rmk:cat-1dim}
	 The \cat{-1} dimension of a group is at least the geometric dimension, but they may not be equal; Brady and Crisp \cite{bradycrisp07} give examples of groups with geometric dimension 2 but \cat{-1} dimension 3.
\end{remark}

Note that any group with finite geometric dimension is torsion-free (see \cite[Appendix F.3]{davis08}). In particular, a group with finite \cat{-1} dimension is in fact freely \cat{-1}, and a group has \cat{-1} dimension 2 if and only if it is the fundamental group of a compact negatively curved simplicial 2-complex---the main object of study in this paper. 

\subsection{Hyperbolic triangles, quadrilaterals and annuli}\label{ss:hyps}

\begin{remark}
	For any angle $0<\theta<\pi$, any $a,b>0$, and any $k \leq 0$, there exists a 2-simplex in $M^2_k$ with angle $\theta$ between two sides of length $a$, $b$. For brevity, we will sometimes refer to such a simplex as an \emph{$\left( a,\theta,b\right)$-fin}.
\end{remark}

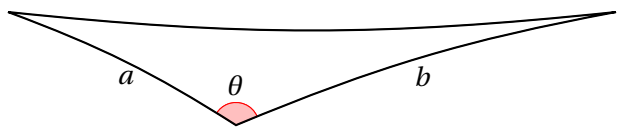
\begin{figure}[htbp]
	\label{fig:fin}
	\centering
	\begin{tikzpicture}
	\coordinate (V1) at (0,1.5);
	\coordinate (V2) at (8,1.5);
	\coordinate (V3) at (3,0);
	\draw[very thin]
	(V1) to[bend right=6] (V2)  to[bend right=6] coordinate[pos=0.93] (A) node[below] {$b$} (V3) coordinate(V) node[above=8pt] {$\theta$} to[bend right=6] coordinate[pos=0.07] (B) node[below] {$a$}  (V1); 
	\tkzMarkAngle[draw=red, fill=red!25, size=3mm](A,V,B);
	\draw[thick]
	(V1) to[bend right=6] (V2)  to[bend right=6] (V3) to[bend right=6] (V1);
	\end{tikzpicture}
	\caption{An $\left( a,\theta,b\right)$-fin.}
\end{figure}

\begin{definition}
	A \emph{Lambert quadrilateral} is a quadrilateral in $M^2_k$ (for $k<0$) with three angles equal to $\pi/2$, as illustrated in Figure \ref{fig:lambert}. The bottom edge, called the \emph{base}, has length $a$, and the top edge, called the \emph{summit}, has length $c$. The single angle $\theta$ not equal to $\pi/2$ is called the \emph{summit angle}. %For the sake of brevity, we say such a Lambert quadrilateral has \emph{parameters $a$, $c$ and $\theta$}.
\end{definition}

\begin{figure}[htbp]
	\centering
	\begin{tikzpicture}
	\coordinate (V1) at (0.5,0);
	\coordinate (V2) at (0,4);
	\coordinate (V3) at (4,5);
	\coordinate (V4) at (3,0);
	\draw[thick]
	(V1) to[bend right=5] coordinate[pos=0.1] (d1) coordinate[pos=0.9] (d2) (V2) to[bend right=5] coordinate[pos=0.1] (c1) coordinate[pos=0.9] (c2) (V3) to[bend right =5] coordinate[pos=0.1] (b1) coordinate[pos=0.9] (b2) (V4) to[bend right=5] coordinate[pos=0.1] (a1) coordinate[pos=0.9] (a2) (V1);
	\tkzMarkRightAngle(a2,V1,d1);
	\tkzMarkRightAngle(d2,V2,c1);
	\tkzMarkRightAngle(b2,V4,a1);
	\tkzMarkAngle[draw=red, fill=red!25, size=7mm,label={$\theta$}](c2,V3,b1);
	\draw[thick]
	(V1) to[bend right=5] (V2) to[bend right=5] node[above]{$c$} (V3) to[bend right =5] (V4) to[bend right=5] node[below]{$a$} (V1);
	\end{tikzpicture}
	\caption{A Lambert quadrilateral}	
	\label{fig:lambert}
\end{figure}
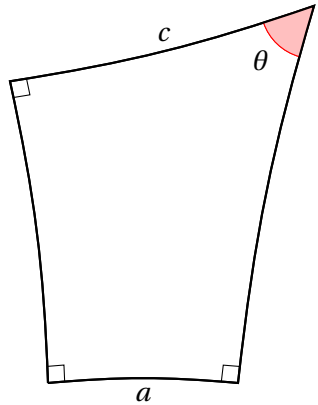

\begin{lemma}\label{lemma:lambertcosh}
	In a Lambert quadrilateral with $k=-1$, \[\sin  \theta  = \frac{\cosh a}{\cosh c}, \]
	where labels are as in Figure \ref{fig:lambert}.
\end{lemma}

\begin{proof}
	We will give a short proof using the hyperbolic sine and cosine rules. Many such identities can also be found in \cite[Chapter 7]{beardonbook}, proved using the definition of the hyperbolic metric on the upper half plane.
	
	Divide the quadrilateral along diagonals, and label as follows:
	\begin{center}
	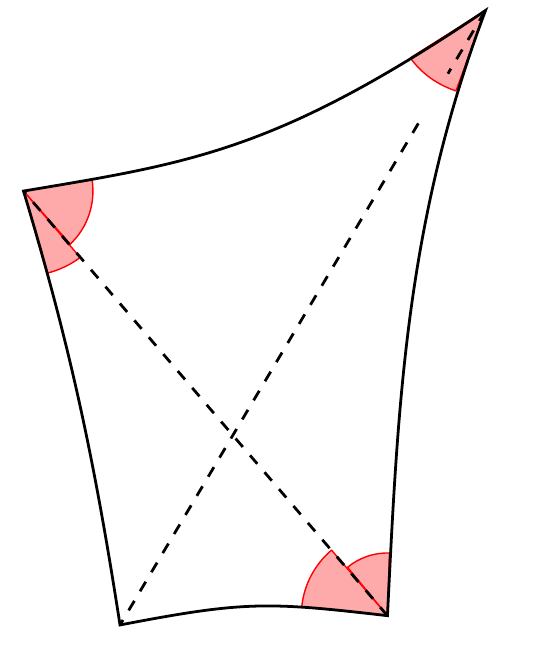
	\end{center}
	Note that $\beta+\beta'=\pi/2=\gamma+\gamma'$, and hence $\cos \gamma = \sin \gamma'$, and $\cos \beta' = \sin \beta$. Firstly, applying the (hyperbolic) Pythagorean theorem in triangles $PQR$ and $PSR$ gives \begin{align} \cosh a \cosh b = \cosh x = \cosh c \cosh d. \label{eqpyth} \end{align}
	The cosine rule in triangle $QRS$ gives \begin{align} \cosh c &= \cosh b \cosh y - \sinh b \sinh y \cos \gamma \nonumber \\
	&= \cosh b \cosh y - \sinh b \sinh y \sin \gamma'. \label{eqfirst} \end{align}
	The sine rule in the right-angled triangle $PQS$ gives \begin{align}
	\sinh y \sin \gamma' = \sinh d, \nonumber \end{align}
	and substituting this into \eqref{eqfirst} gives \begin{align} \cosh c
	= \cosh b \cosh y - \sinh b \sinh d. \label{eqcfirst} \end{align}
	Similarly applying the cosine rule to triangle $PQS$ and substituting for $\cos \beta' = \sin \beta$ using the sine rule in triangle $QRS$, we obtain: \begin{align} \cosh a
	= \cosh d \cosh y - \sinh b \sinh d \sin \theta. \label{eqsecond} \end{align}
	Now, from \eqref{eqpyth}, we have \begin{align} \cosh d
	= \frac{\cosh b \cosh a}{\cosh c}, \nonumber \end{align} and substituting this into \eqref{eqsecond} gives \begin{align}
	& &\cosh a &= \frac{\cosh b \cosh a \cosh y}{\cosh c} - \sinh b \sinh d \sin \theta. \nonumber \\
	&\implies \quad &\cosh c &= \cosh b \cosh y -\sinh b \sinh d \sin \theta \left( \frac{\cosh c}{\cosh a} \right). \label{eqcsecond} \end{align}
	Finally, equating \eqref{eqcfirst} and \eqref{eqcsecond} we obtain: \begin{align*} & &\cosh b \cosh y - \sinh b \sinh d &= \cosh b \cosh y -\sinh b \sinh d \sin \theta \left( \frac{\cosh c}{\cosh a} \right) \\ &\implies \quad & 1 &= \sin \theta \left( \frac{\cosh c}{\cosh a} \right) \\ &\implies \quad &\sin \theta &= \frac{\cosh a}{\cosh c}, \end{align*} as required.
\end{proof}

\begin{lemma}\label{lemma:lambertexist}
	A Lambert quadrilateral in $\mathbb{H}=M^2_{-1}$ with summit length $c$ and summit angle $\theta$ exists if and only if $\theta_c < \theta < \pi/2$, where\[
	\theta_c = \sin^{-1} \left( \frac{1}{\cosh c} \right). \]
\end{lemma}

\begin{proof}
	Recall that for any two non-crossing geodesics in $\mathbb{H}$, there is a unique geodesic perpendicular to both. Now, construct a geodesic segment $C$ of length $c$ in $\mathbb{H}$, with a perpendicular geodesic $B$ at one end, and a perpendicular geodesic $B'$ at the other end. Clearly $B'$ and $B$ do not cross, but no (non-degenerate) Lambert quadrilateral exists with summit $C$, since the unique geodesic perpendicular to $B$ and $B'$ is $C$. Now continuously decrease the angle $\theta$ between $C$ and $B'$. By convexity of the metric, the unique geodesic perpendicular to $B$ and $B'$ will form a Lambert quadrilateral on the same side of $C$ as the angle $\theta$, until $\theta$ reaches the value $\theta_c$ at which $B$ and $B'$ cross. This can be calculated (for example) by setting $a=0$ in the formula from Lemma \ref{lemma:lambertcosh}: \[
	\theta_c = \sin^{-1} \left( \frac{1}{\cosh c} \right), \]
	as required.
\end{proof}

Note that, given such $c$ and $\theta$, the quadrilateral is then uniquely determined.

\begin{lemma}\label{lemma:lambertratio}
	For any $0<\theta<\pi$, and for any $a$, $c$ such that $c>a>0$, there exists $k<0$ and a Lambert quadrilateral in $M^2_k$ with base length $a$, summit length $c$ and summit angle greater than $\theta/2$.
\end{lemma}

\begin{proof}
	First, let \[
		\bar{c}=\cosh^{-1} \left( \frac{1}{\sin \left( \frac{\theta}{2} \right)} \right), \]
	so that $\frac{\theta}{2}=\theta_{\bar{c}}$. For any $\bar{a}<\bar{c}$, we have \[
	1>	\frac{\cosh \bar{a}}{\cosh \bar{c}}>\frac{1}{\cosh \bar{c}}=\sin \left( \frac{\theta}{2} \right)
		\]
	and hence:\[
	\frac{\pi}{2}>	\sin^{-1} \left( \frac{\cosh \bar{a}}{\cosh \bar{c}} \right) > \frac{\theta}{2} 
	\]	
	Applying Lemma \ref{lemma:lambertexist}, we see that a Lambert quadrilateral exists in $M^2_{-1}$ with base length $\bar{a}$, summit length $\bar{c}$ and summit angle $\sin^{-1} \left( \frac{\cosh \bar{a}}{\cosh \bar{c}} \right) > \theta/2$. By multiplying the metric by a factor $c/\bar{c}$, we obtain a Lambert quadrilateral in $M^2_k$, where $k=-\left( \bar{c}/c \right)^2$, satisfying the required conditions.
\end{proof}

\begin{lemma}\label{lemma:annulus}
	For any $0<\theta<\pi$, and any $A$, $C$ such that $C>A>0$, there exists $k<0$ and a locally \cat{k} annulus with one locally geodesic boundary component of length $A$, and one boundary component of length $C$ which is locally geodesic everywhere except for one point where it subtends an angle greater than $\theta$.
\end{lemma}

\begin{proof}
	Apply Lemma \ref{lemma:lambertratio} with the same $\theta$, $a=A/2$ and $c=C/2$. Now take two copies of the resulting Lambert quadrilateral, and glue together two pairs of sides to obtain the required annulus (see Figure \ref{fig:annulus}).
\end{proof}

%We will refer to the annulus described in Lemma \ref{lemma:annulus} as an \emph{$\left(A,C,\theta\right)$-annulus}---note that, unlike the notation for a fin, $\theta$ is only a \emph{lower bound} for the angle present in the annulus.

\begin{figure}[htbp]
	\centering
	\input{Annulus}
	\caption{Gluing two Lambert quadrilaterals to obtain an annulus. Identify the two sides of the left picture. The bottom edge of the annulus has length $A$, the top edge has length $C$, and the angle in the top edge is $>\theta$.}	
	\label{fig:annulus}
\end{figure}
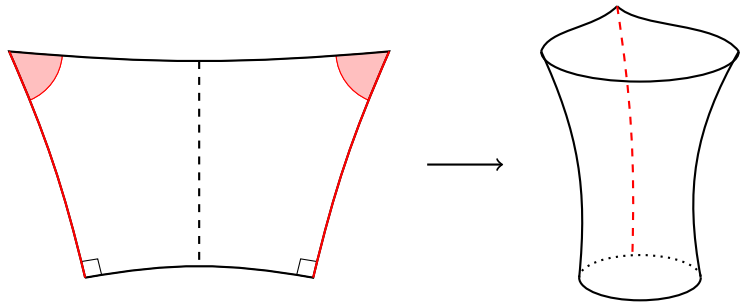

\subsection{Transversality}

As discussed in the introduction, we would like to build new negatively curved complexes by gluing fins and annuli to existing ones. To avoid introducing positive curvature, for example by identifying a pair of adjacent sides in two fins, we would like to glue along paths that intersect transversely. The following lemma ensures that this can always be arranged:

\begin{definition}
	Let $\gamma = \left( e_1 \gdots e_n \right)$ and $\gamma' = \left( e'_1 \gdots e'_{n'} \right)$ be two closed geodesics in $K$. We say $\gamma$ and $\gamma'$ \emph{intersect transversely} if $e_i \neq e'_j$ and $\bar{e}_i \neq e_j$ for all $i$ and $j$ (that is, the two loops do not share an edge). Similarly, we say $\gamma$ \emph{intersects itself transversely} if $e_i \neq e_j$ and $\bar{e}_i \neq e_j$ for all $i \neq j$.
\end{definition}

\begin{lemma}\label{lemma:transverse}
	Let $K$ be a negatively curved simplicial 2-complex of curvature $k$. Let $\left\{ \gamma^{(1)} \gdots \gamma^{(m)} \right\}$ be a collection of simplicial closed geodesics such that the corresponding cyclic subgroups of $\pi_1\left(K\right)$ form a malnormal family. Then there is a 2-complex $\bar{K}$ such that: \begin{itemize}
		\item $\bar{K}$ is a negatively curved simplicial complex of curvature $\bar{k}$, where $k\leq\bar{k}<0$;
		\item there is an inclusion $i \colon K \rightarrow \bar{K}$ and a deformation retraction $r \colon \bar{K} \rightarrow K$;
		\item there are closed geodesics $ \bar{\gamma}^{(1)} \gdots \bar{\gamma}^{(m)}$ in $\bar{K}$, such that $r \left( \bar{\gamma}^{(i)} \right)=\gamma^{(i)}$, which intersect themselves and each other transversely.
	\end{itemize}
\end{lemma}

%CAREFUL OF EDGE ORIENTATIONS IN THE BELOW

\begin{proof}
The complex $\bar{K}$ will be obtained from $K$ by gluing on fins to shorten the intersection between the closed geodesics $\gamma^{(i)}$. We will ensure that there is ``room'' to glue these fins by repeatedly taking the comparison complex and obtaining an excess angle at the vertices. 
	
Let $\gamma^{(i)} = \left( e^{(i)}_1 \gdots e^{(i)}_{n(i)} \right)$. Let $I$ be the number of repetitions (ignoring orientation) in the list
	\[\left\{e^{(i)}_j  \; \middle| \; i=1 \gdots m, \; j=1 \gdots n(i)	\right\};\]
that is, $I$ counts $d-1$ for each edge of $K$ that occurs $d$ times in the union of the $\gamma^{(i)}$. Thus, $I$ counts the number of failures of transversality of the set of geodesics; if $I=0$, then the $\gamma^{(i)}$ intersect themselves and each other transversely, and the conclusions of the lemma hold already. To prove the lemma, we will show that if $I > 0$, it is always possible to find a homotopy equivalent negatively curved simplicial 2-complex with $I$ reduced by $1$. 
	
So, suppose $I > 0$. The first stage is to replace $K$ by a comparison complex $K'$ of curvature (say) $k/2$. Trivially, the inverse of the isomorphism $' \colon K \rightarrow K'$ is a homotopy equivalence and preserves the $\gamma^{(i)}$, so we may safely proceed with $K'$ instead of $K$. Let $\delta$ be the excess angle $\delta(K,K')$; if Shapes($K$) is not finite, then we may ensure $\delta>0$ by taking the minimum in Definition \ref{def:delta} only over the finitely many angles at vertices contained in the $\gamma^{(i)}$ (see Remark \ref{rmk:deltafinite}).
	
	%BOX THIS OFF INTO A LEMMA
	%CAN WE SAY THIS QUICKER WITH A UNIVERSAL COVER ARGUMENT?
	
	Since $I > 0$, we may assume that there are two closed geodesics $\gamma = \left\{ e_1 \gdots e_r \right\}$, $\gamma' = \left\{ e'_1 \gdots e'_s \right\}$, obtained by relabelling and possibly reversing the $\gamma^{(i)}$, such that $e_1=e'_1$ with orientation (note that $\gamma$ and $\gamma'$ may be distinct relabellings or reversals of the same $\gamma^{(i)}$). Without loss of generality, $r \leq s$. 
	
	%Denote by $g$ and $g'$ the elements of $\pi_1\left(K',v_0=\iota(e_1)\right)$ corresponding to the based loops $\gamma$ and $\gamma'$ respectively. Consider the universal covering action of $\pi_1\left(K',v_0\right)$ on $\widetilde{K}'$. We claim that the	
	
	Now, suppose that, for all $z \in \mathbb{Z}$, $e_{z\,\,\,\text{mod}\,\,r}=e'_{z\,\,\,\text{mod}\,\,s}$. In the case that $\gamma$ and $\gamma'$ are relabellings of different $\gamma^{(i)}$, this implies that the loops are both powers of a common loop, contradicting the malnormal family assumption. If $\gamma$ and $\gamma'$ are relabellings without reversal of the same $\gamma^{(i)}$, then this again implies that $\gamma^{(i)}$ is a proper power of a loop, which contradicts the assumption. The only remaining case is that $\gamma$ and $\gamma'$ are relabellings of the same $\gamma^{(i)}$, but one of them is reversed; that is, $e'_1$ $(=e_1)$ $=\bar{e}_t$ for some $t>1$. It follows that $e_2=\bar{e}_{t-1}$, $e_3=\bar{e}_{t-2}$ and so on---hence, $\gamma$ must contain either an edge $e=\bar{e}$ (which is forbidden by definition) or an adjacent pair $e, \bar{e}$, which contradicts the fact that the loops are local geodesics. %Should the \bar's in here be \overline's?
	
	It follows that there is some $j<s$ such that $e_{j}=e'_j$ but $e_{j+1} \neq e'_{j+1}$. Let $a= \left|e_j\right|$ and $b=\left|e'_{j+1}\right|$. Now, form a complex $K^+$ from $K'$ by gluing an $(a,\pi-\delta, b)$-fin along $e_j$ and $e'_{j+1}$ (see Figure \ref{fig:gluefin}). For any closed geodesic $\ell$ in $K'$ containing $\left( e'_j, e'_{j+1} \right)$, the loop $\ell^+$ in $K^+$ obtained from $\ell$ by replacing $\left( e'_j, e'_{j+1} \right)$ with the new edge $e'$ (along the top of the fin) is a local geodesic in $K^+$, and there is a clear deformation retraction to $K'$ sending $\ell^+$ to $\ell$. So, consider $K^+$ together with the set of closed geodesics obtained from $\left\{ \gamma^{(1)} \gdots \gamma^{(m)} \right\}$ by replacing any occurrence of $\left( e'_j, e'_{j+1} \right)$ with $e'$. This satisfies the first two properties in the lemma, and the value of $I$ is strictly lower than for the original set of loops and the complex $K$ (we have removed at least one edge in the intersection between $\gamma$ and $\gamma'$). It therefore remains only to check that $K^+$ is negatively curved.
	
	For this, we use the link condition. The only three links to check are those at $v_{j-1}$, $v'_{j+1}$ and $v_j$, since all the others are unchanged from $K'$. To obtain $\text{link} \left(v_{j-1}, K^+\right)$ and $\text{link} \left(v'_{j+1}, K^+\right)$, we have simply glued a leaf to the corresponding links in $K'$, and to obtain $\text{link} \left(v_{j}, K^+\right)$, we have (in light of Remark \ref{rmk:angleexcess}), glued an edge of length $\pi-\delta$ between two vertices of distance $ \geq \pi+2\delta$ in $\text{link} \left(v_{j}, K'\right)$. Neither of these can introduce a failure of the \cat{1} condition (by Remark \ref{rmk:cat1}) and so $K^+$ is negatively curved. This completes the proof.
	\end{proof}

\begin{figure}[htbp]
	\centering
	\input{Transverse}
	\caption{Gluing on a fin}	
	\label{fig:gluefin}
\end{figure}
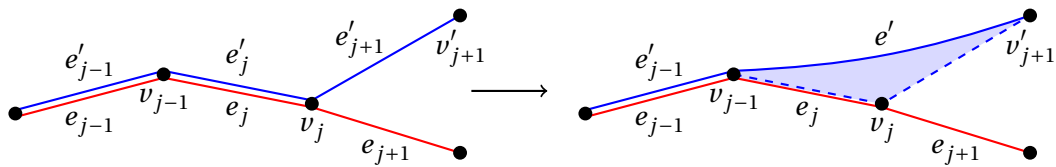

\begin{remark}
	It is not difficult to see that we may improve the above lemma to make the geodesics not only transverse, but disjoint. Indeed, if two geodesics intersect transversely at a vertex $v$, we can make them locally disjoint by gluing a fin along one of the two geodesics at $v$. Alternatively, it can be seen as an application of Theorem \ref{thm:main} in the next section, where the underlying graph is a star with vertex spaces $K$ for the central vertex, circles for the leaves, and circles for each edge space attached to each of the loops $\gamma^{(i)}$.
\end{remark}

\section{The gluing theorem}\label{sec:mainthm}

%MAYBE THE WHOLE THEOREM CAN BE STATED IN TERMS OF GRAPHS OF GROUPS

We are now ready to state our main gluing theorem. As explained in the subsequent remarks, the statement below is not the strongest available; however, it allows for a straightforward application to the limit group situation (see section \ref{sec:lg} for more details).

\begin{theorem}\label{thm:main}
	Let $X$ be a graph of spaces with finite underlying graph $\Gamma$, such that: \begin{enumerate}
		\item Vertex spaces are one of three types: \begin{description}
			\item[type P] single points,
			\item[type N] simplicial circles, or
			\item[type M] connected negatively curved simplicial 2-complexes that are neither points nor circles.
		\end{description} \label{item:vs}
		\item Edge spaces are either circles or points.
		\item Each circular edge space connects a type N vertex space to a type M vertex space. \label{item:endcircle}
		\item The images of attaching maps of circular edge spaces are (simplicial) closed geodesics.\label{item:li}% %IS THIS REALLY AN EQUIVALENCE?
		\item For each type M vertex group in the corresponding graph of groups, the family of subgroups corresponding to incident edge groups is a malnormal family.\label{item:malfam}
	\end{enumerate}
	Then $X$ is homotopy equivalent to a negatively curved simplicial 2-complex $\bar{X}$. Moreover, if $X$ has compact vertex spaces, then $\bar{X}$ is compact.
\end{theorem}

\begin{remark}\label{rmk:thmconds}
 Firstly, note that all negatively curved simplicial 2-complexes are homotopy equivalent to one of the three types of vertex space listed, and so condition \ref{item:vs} requires nothing more than that vertex spaces are homotopy equivalent to negatively curved simplicial 2-complexes. Condition \ref{item:endcircle} in fact requires only that two type N vertex spaces are not connected by a circular edge space, since if a circular edge space connected two type M vertex spaces we could then subdivide the corresponding edge of $\Gamma$ and insert a circular vertex space in between. Condition \ref{item:li} is always achievable; indeed, if $K$ is a locally \cat{-1} space, then every conjugacy class in $\pi_1(K)$ is represented by a unique closed geodesic (see \cite{bh}). If $K$ is a simplicial complex, then by subdivision we may assume this closed geodesic is simplicial. Conjugacy classes in $\pi_1(K)$ correspond to free homotopy classes of unbased loops in $K$, and hence each attaching map is homotopic to an isometry to a simplicial closed geodesic. 
\end{remark}	
	% %CHECK THE CONDITIONS FOR THIS TO HOLD
\begin{remark}
Condition \ref{item:malfam} can also be unravelled a little. As mentioned above, in each vertex space, the loops corresponding to circular edge spaces determine conjugacy classes in the fundamental group. As in Lemma \ref{lemma:transverse}, the malnormal family assumption then says that, whichever representatives of these conjugacy classes are chosen when defining the graph of groups, the corresponding edge subgroups have trivial intersection. The other requirement is that the subgroups are individually malnormal; which, for cyclic subgroups of a torsion-free \cat{-1} group, is equivalent to requiring that they are maximal infinite cyclic---that is, not generated by a proper power. This is because torsion-free \cat{-1} groups cannot contain any Baumslag-Solitar subgroup (see \cite{bh}).
\end{remark}

\begin{proof}[Proof of Theorem \ref{thm:main}] \begin{sloppypar}
	To begin, we consider $X$ as a topological graph of spaces, with specified metric on the type M vertex spaces only. We will describe how to metrize the rest of the vertex and edge spaces in the proof.
	
	Consider a type M vertex space $X_v$ of $X$. Note that $X_v$ together with the set $\left\{ \partial_e \left( X_e \right) \; \middle| \; e \in E(\Gamma), \; \iota(e)=v, \; X_e \simeq \mathbb{S}^1 \right\}$ of images of incident circular edge spaces satisfies the conditions of Lemma \ref{lemma:transverse}. We may therefore find a negatively curved 2-complex $\bar{X}_v$, equipped with a deformation retraction $r_v \colon \bar{X}_v \rightarrow X_v$, inclusion $i_v \colon X_v \rightarrow \bar{X}_v$, and a transverse set of loops $\left\{ \gamma_e \middle| \; e \in E(\Gamma), \; \iota(e)=v, \; X_e \simeq \mathbb{S}^1 \right\}$ such that $r_v \left( \gamma_e \right) = \partial_e \left( X_e \right)$. Without loss of generality (by taking a comparison complex if necessary) we may assume that $\bar{X}_v$ has an excess angle $\delta_v>0$ around each loop $\gamma_e$ (as in Remark \ref{rmk:angleexcess}) so that the angle subtended through each vertex by each loop $\gamma_e$ is at least $\pi+2\delta_v$. 
\end{sloppypar}	
	Now consider a type N vertex space $X_w$. For each circular incident edge space $X_e$ with $\tau(e)=w$, the attaching map $\partial_{\bar{e}}$ is a $d$ to 1 covering map for some $d=d(e)$. Choose a positive number $a_w$ satisfying: \[
	a_w< \min \left\{ \frac{l \left( \gamma_e \right)}{d(e)} \;\middle|\;e \in E(\Gamma), \; \tau(e)=w  \right\}.
	\]
		
We will now describe how to modify the complex $X$ into the complex $\bar{X}$. By Lemma \ref{lemma:goshe}, the two complexes will be homotopy equivalent.
	\begin{enumerate}
		\item Replace each type M vertex space $X_v$ with $\bar{X}_v$ as above, then replace each attaching map $i_v \circ \partial_e$ with a homotopic map to the closed geodesic $\gamma_e$. 
		\item Metrize each type N vertex space $X_w$ to have length $a_w$.
		\item For each circular edge space $X_e$, by condition \ref{item:endcircle}, we may assume $\iota(e)=v$ and $\tau(e)=w$ where $X_w$ is type N. After the above, the two ends of $X_e \times [0,1]$ are identified with closed geodesics of length $l \left( \gamma_e \right)$ and $a_w d(e)$. Since $a_w d(e) < l \left( \gamma_e \right)$, there exists by Lemma \ref{lemma:annulus} a negatively curved annulus $A_e$ with one geodesic boundary component of length $a_w d(e)$, and one boundary component of length $l \left( \gamma_e \right)$ subtending an angle greater than $\pi-\delta_{v}$ at a vertex and geodesic elsewhere. Therefore, we may replace $X_e \times [0,1]$ with (a copy of) $A_e$ such that the metrics on $A_e$, $X_v$ and $X_w$ are all compatible (ensuring we position the vertex of $A_e$ at some vertex of $X_v$).
		\item Triangulate $A_e$ to ensure it is simplicial.
		\item For each type P vertex space $X_v$, take a point $\bar{X}_v$. 
		\item For each edge $e$ such that $X_e$ is a point, attach a line between $\bar{X}_{\iota(e)}$ and $\bar{X}_{\tau(e)}$. Note that the endpoints of this line (that is, the attaching maps), as well as its length, are irrelevant from the perspective of homotopy, since the vertex spaces are path connected.
	\end{enumerate}		
		
\begin{comment}		
	Now the complex $\bar{X}$ may be built from $X$ as follows. 
	\begin{enumerate}
		\item Replace each type M vertex space $X_v$ with $\bar{X}'_v$ as above.
		\item Replace each type N vertex space $X_v$ with an oriented circle $\bar{X}_v$ of length $a_v$.
		\item For each circular edge space $X_e$, by condition \ref{item:endcircle}, we may assume $\tau(e)=v$ where $X_v$ is type N. Replace the annulus $X_e \times [0,1]$ with an $(a_v d(e), l \left( \gamma_e \right), \pi-\delta_{\iota(e)} )$-annulus $A_e$. Glue the non-geodesic boundary component of $A_e$ to $\gamma_e$, gluing the vertex to some vertex of $\bar{X}_w$. This map can be taken to be a local isometry everywhere except the vertex of $A_e$. Glue the geodesic boundary component of $A_e$ by a locally isometric $d(e)$-to-1 map to the circle $\bar{X}_v$, ensuring that the orientation matches that of $X_e$. 
		\item Subdivide $A_e$ to make it simplicial, and to make the above two maps simplicial.
		\item For each type P vertex space $X_v$, take a point $\bar{X}_v$. 
		\item For each edge $e$ such that $X_e$ is a point, attach a line between $\bar{X}_{\iota(e)}$ and $\bar{X}_{\tau(e)}$. Note that the endpoints of this line are irrelevant from the perspective of homotopy, since the vertex spaces are path connected.
	\end{enumerate}
	\end{comment}
	
	Topologically, to obtain $\bar{X}$ from $X$ we have simply replaced edge cylinders and vertex spaces with homotopy equivalent spaces, and attaching maps with homotopic maps. It follows from Lemma \ref{lemma:goshe} that $\bar{X}$ and $X$ are homotopy equivalent. Note that $\bar{X}$ is still, topologically, a graph of spaces, and it supports a global metric, but this is no longer the standard metric on a graph of spaces as described in section \ref{ss:mgos}.
	
	We must now check that $\bar{X}$ is a negatively curved simplicial 2-complex. By construction, it is a simplicial 2-complex, all of whose simplices are negatively curved. It is sufficient, therefore, to check the link condition on vertices. We already have that each vertex space $\bar{X}_v$, and each annulus $A_e$, is negatively curved, so it suffices to check the link condition for vertices at which the annuli $A_e$ are glued to vertex spaces.
	
	For vertices in type N vertex spaces, links consist of two vertices connected by several (possibly subdivided) arcs, of length $\pi$. These satisfy the link condition, since all circles in the space have length at least $2\pi$.
	
	Now let $\bar{X}_v$ be a type M vertex space, and let $x$ be a vertex of $\bar{X}_v$ contained in at least one $\gamma_e$. To obtain $\text{link}(x,\bar{X})$ from $\text{link}(x, \bar{X}_v)$, we glue on a number of arcs of length $>\pi-\delta_v$. We may prove by induction that the resulting space remains \cat{1}.
	
	Firstly, recall from above that the geodesics $\gamma_e$ subtend angles of at least $\pi+2\delta_v$ at each vertex. In particular, the first arc glued on to $\text{link}(x, \bar{X}_v)$ connects two points of distance at least $\pi+2\delta_v$, and is itself of length $>\pi-\delta_v$; Remark \ref{rmk:cat1} then says that the space remains \cat{1}.
	
	For subsequent arcs, let us assume for induction that after gluing on $r-1$ arcs, the space is still \cat{1}. Now suppose the $r$th arc is to be glued between points $a$ and $b$. By Lemma \ref{lemma:transverse}), $a$ and $b$ are distinct from any previous points to which arcs have been glued, and hence any path between $a$ and $b$ which is not contained in $\text{link}(x, \bar{X}_v)$ must begin and end with an edge from $\text{link}(x, \bar{X}_v)$. All such edges have length at least $\delta_v$, and hence any such path has length at least $2\delta_v + \pi-\delta_v=\pi+\delta_v$. On the other hand, any path connecting $a$ and $b$ which does lie in $\text{link}(x, \bar{X}_v)$ is of length at least $\pi+2\delta_v$ as noted before. Thus Remark \ref{rmk:cat1} applies again, and so the space again remains \cat{1}. 
	
	Therefore, the link condition holds for $\bar{X}$, and so $\bar{X}$ is a negatively curved simplicial complex. Moreover, it is clear by construction that if all type M vertex spaces of $X$ are compact, then $\bar{X}$ is compact. This completes the proof.
\end{proof}

\section{Limit groups}\label{sec:lg}

As discussed in the Introduction, the motivating application for Theorem \ref{thm:main} was to prove Theorem \ref{thm:limcat-1} below:

\begin{theorem}\label{thm:limcat-1}
	Let $G$ be a limit group. Then $G$ is \cat{-1} if and only if $G$ is $\delta$-hyperbolic.
\end{theorem}

This is a simplified version of Theorem \ref{thm:equivalence}, which will follow quickly from one of the defining characterisations of a limit group (see Theorem \ref{thm:limgpice}), but before we launch into this we will give some context as to the relevance and usefulness of limit groups. To this end, we begin by stating the simplest definition:

\begin{definition}\label{def:wrf}
A \emph{limit group} (also \emph{finitely generated fully residually free group}) is a finitely generated group $G$ such that, for any finite subset $1 \notin S \subset G$, there is a homomorphism $h \colon G \rightarrow \mathbb{F}$ to a non-abelian free group $\mathbb{F}$ such that $h$ is injective on $S$.
\end{definition}

Limit groups were first introduced in the study of equations over free groups; their importance lies in the fact that, in the study of the sets of solutions of these equations, limit groups correspond to irreducible varieties.
%Limit groups were first introduced in the study of equations over free groups and one interesting interpretation of them in this light can be seen from the above definition alone. Let $G$ be some finitely generated (of rank $r$, say) non-free group, and let $\mathbb{F}$ be a free group of rank $\geq2$. Consider the set $\text{Hom}\left(G,\mathbb{F} \right)$. This is an algebraic variety in $\mathbb{F}^r$---indeed, homomorphisms from $G$ to $\mathbb{F}$ correspond precisely to those $r$-tuples of elements of $\mathbb{F}$ satisfying all the defining relations of $G$. Elements $g\in G$ correspond to functions $\phi_g \colon \text{Hom}\left(G,\mathbb{F} \right) \rightarrow \mathbb{F}$ defined by $\phi_g \left( f \right) = f(g)$. Now let $G$ be a limit group. Then, given any finite set $g_1 \gdots g_n \in G$, either one of the $g_i=1$ or there exists a homomorphism $h \colon G \rightarrow \mathbb{F}$ such that $h\left(g_i\right) \neq 1$ for all $i$. Therefore, if for every $f \in \text{Hom}\left(G,\mathbb{F} \right)$ we have some $i$ such that $\phi_{g_i} \left( f \right) = 1$, it follows that there exists $i$ such that $g_i=1$; or in other words, $\ker \left( \phi_{g_i} \right)=\text{Hom}\left(G,\mathbb{F} \right)$. This is exactly the definition of an irreducible algebraic variety; hence, limit groups correspond to irreducible varieties over free groups.
The original definition (motivating the name) was given by Sela \cite{sela2001}, and is quite different (although equivalent) to Definition \ref{def:wrf}. We refer the interested reader to \cite{sela2001}, alongside \cite{bestvinafeighn09} and \cite{wilton09}, for more information and references. 

We will use several facts concerning limit groups without proof, and we state these below for convenience. Properties \ref{item:limgptf}, \ref{item:limsg} and \ref{item:csa} are easy to see from the definition, and proofs of all the others can be found in \cite{sela2001} or \cite{bestvinafeighn09} except where otherwise indicated.

\begin{theorem}[Properties of limit groups]\label{thm:limgpprops}
	Let $G$ be a limit group. Then: \begin{enumerate}
			\item $G$ is torsion-free.\label{item:limgptf}
			\item \emph{\cite[Corollary 4.4]{sela2001}} $G$ is finitely presented. \label{item:limfp}
			\item Every finitely generated subgroup of $G$ is a limit group. \label{item:limsg}
			\item  \emph{\cite[Corollary 4.4]{sela2001}} Every abelian subgroup of $G$ is finitely generated.\label{item:limabfg}
			\item \emph{\cite[Lemma 1.4]{sela2001}} Every nontrivial abelian subgroup of $G$ is contained in a unique maximal abelian subgroup.\label{item:csa}
			\item \emph{\cite[Lemma 1.4]{sela2001}} Every maximal abelian subgroup of $G$ is malnormal.\label{item:csaproper}
			\item  \emph{\cite[Lemma 2.1]{sela2001}} If $G=A *_C B$ for $C$ abelian, then any non-cyclic abelian subgroup $M \subset G$ is conjugate into $A$ or $B$.\ \label{item:noncycabconj}
			\item \emph{\cite{alibegovicbestvina06}} $G$ is \cat{0}.\label{item:limgpscat0}
		\end{enumerate}
\end{theorem}

Much of the work that has been done on limit groups, including the proof of \ref{item:limgpscat0} above, depends upon a powerful structure theory. For full details of this (often expressed in terms of \emph{constructible limit groups}), the reader is referred again to \cite{sela2001} and subsequent papers in that series, as well as the expositions \cite{bestvinafeighn09, wilton09, champetier05} (see also \cite{kharlampovich98a} and subsequent papers). A particularly elegant result that comes out of this theory is Theorem \ref{thm:limgpice} below, proved in \cite{kharlampovich98} and \cite{champetier05}. To state it, we must first introduce the following notion. Our overview follows that given in \cite{wilton08}.

\begin{definition}
Let $G'$ be a group, let $g \in G'$, and let $C(g)$ denote the centralizer of $g$. Let $n\geq1$. Then the group $G' *_{C(g)} \left( C(g) \times \mathbb{Z}^n \right)$ is called a \emph{centralizer extension of $G'$ by $C(g)$}.
\end{definition}

\begin{definition}
A group is said to be an \emph{iterated centralizer extension} if it is either a finitely generated free group, or can be obtained from one by taking repeated centralizer extensions. The class of iterated centralizer extensions is denoted \ice; sometimes we will refer to a group in \ice as an \emph{\ice group}. An \ice group obtained by taking a free group and then taking a centralizer extension $n$ times is said to have \emph{height $n$}.
\end{definition}

\begin{theorem}[\cite{kharlampovich98, champetier05}]\label{thm:limgpice}
	Limit groups coincide with finitely generated subgroups of \ice groups.
\end{theorem}

\begin{remark}\label{rmk:limgpmaxcyclic}[see also \cite[Remark 1.14]{wilton08}]
	If $G$ is an \ice group (and hence a limit group by Theorem \ref{thm:limgpice}) which is not free, then it follows from property \ref{item:csa} of Theorem \ref{thm:limgpprops} that centralizers in $G$ are abelian. Since $G$ is in \ice, there is an \ice group $G'$ such that $G=G' *_{C(g)} \left( C(g) \times \mathbb{Z}^n \right)$. If $C(g)$ is non-cyclic, then by property \ref{item:noncycabconj}, $C(g)$ is conjugate into one of the two components in the iterated centralizer decomposition of $G'$; by induction, it follows that $C(g)$ is conjugate into some previously attached $ C(g') \times \mathbb{Z}^m$. At this stage, we could therefore have attached $ C(g') \times \mathbb{Z}^{m+n}$ instead. It follows that, when building limit groups, we may assume all centralizer extensions are by (infinite) cyclic centralizers. 
\end{remark}

A consequence of Theorem \ref{thm:limgpice} is that there is a natural graph of spaces decomposition for any \ice group $G$. If $G$ is free, it is a single vertex space: a compact graph with fundamental group $G$. Otherwise, $G=G' *_{\left< g \right>} \left( \left< g \right> \times \mathbb{Z}^n \right)$, where we can assume (by induction) that $G'=\pi_1(Y')$ for some graph of spaces $Y'$. Then $G$ has a graph of spaces decomposition $Y$ with underlying graph an edge, one vertex space $M=Y'$, one vertex space $N$ isomorphic to the $n+1$-torus $\mathbb{T}^{n+1}$, and edge space a circle $A$. The attaching maps send $A$ to a closed curve $\gamma \subset M$ representing $g$ in $\pi_1(M)=G'$, and to a coordinate circle of $N=\mathbb{T}^{n+1}$. We may assume without loss of generality that $M$ and $N$ are simplicial, the attaching maps are combinatorial local isometries, and their images are simplicial closed geodesics.

\begin{remark}\label{rmk:limgpgos}
	The above graph of spaces decomposition for an \ice group $G$ induces a graph of spaces decomposition for any subgroup $H<G$. Therefore, any limit group $H$ has a graph of spaces decomposition $X$ with edge spaces either circles or points, and vertex spaces covering spaces of either $M$ or $N$. Moreover, since all limit groups are finitely generated, we can assume that $X$ has finite underlying graph. 
\end{remark}

We are now ready to prove the following consequence of Theorem \ref{thm:main}.

\begin{theorem}\label{thm:limnc}
	Let $H$ be a limit group which does not contain any subgroup isomorphic to $\mathbb{Z}^2$, and let $X$ be the graph of spaces induced by an embedding of $H$ into an \ice group. Then there exists a compact negatively curved simplicial 2-complex $\bar{X}$ which is homotopy equivalent to $X$.
\end{theorem}

% %We don't need it to be homotopy equivalent; in fact we just need a map that induces a \pi_1 isomorphism, thanks to Whitehead's Theorem.

%WHY ARE TYPE M VERTEX SPACES NEVER CIRCLES?

\begin{proof}
	The proof uses Theorem \ref{thm:limgpice}, along with Theorem \ref{thm:main} and induction on height. 
	
	Clearly the result holds if $H$ embeds into an \ice group of height 0. So suppose $H$ embeds into an \ice group $G=\pi_1 \left( M *_A N \right)$ of height $n$, and assume the result is proven for all limit groups that do not contain $\mathbb{Z}^2$ and that embed in \ice groups of height $\leq n-1$. The graph of spaces $X$ for $H$ has two types of vertex space: those mapping to $M$ and those mapping to $N$. Call these type M and type N respectively. 
	
	Edge spaces of $X$ are either lines or circles. Since $H$ is finitely generated, this implies that the vertex groups are all finitely generated. Since $H$ does not contain $\mathbb{Z}^2$, each type M vertex group is therefore a limit group that does not contain $\mathbb{Z}^2$, and each type N vertex space is homotopy equivalent to either a point or a circle. By the inductive hypothesis, each type M vertex space is therefore homotopy equivalent to a compact negatively curved simplicial 2-complex. 
	
	By Lemma \ref{lemma:goshe}, we may replace vertex and edge spaces $X_v$, $X_e$ of $X$ with homotopy equivalent spaces $X'_v$, $X'_e$, to obtain a graph of spaces $X'$ with the same homotopy type as $X$. We do this as follows:
	
		\begin{enumerate}
		\item For each edge space $X_e$ which is a line, let $X'_e$ be a point.
		\item For each edge space $X_e$ which is a circle, let $X'_e$ be a circle.
		\item For each type N vertex space $X_v$ which is homotopy equivalent to a point, let $X'_v$ be a point.
		\item For each type N vertex space $X_v$ which is homotopy equivalent to a circle, let $X'_v$ be a simplicial circle.
		\item For each type M vertex space $X_v$, apply the inductive hypothesis to find a homotopy equivalent compact negatively curved simplicial 2-complex $X'_v$.
	\end{enumerate}
	
	As in the proof of Theorem \ref{thm:main}, the attaching maps are defined by composing the attaching maps in $X$ with the homotopy equivalences applied to the vertex spaces, followed by a further homotopy to ensure that the images of attaching maps of circular edge spaces are closed geodesics. That is, once we have fixed an edge and vertex space, we choose as our attaching map a local isometry which represents the corresponding attaching map in the graph of groups. As before, for those edge spaces which are points, this can be any map.
	
	% %do this better as a lemma
	At this stage, $X'$ is a compact graph of spaces satisfying conditions \ref{item:vs} to \ref{item:li} of Theorem \ref{thm:main}. To show that it also satisfies condition \ref{item:malfam}, note that the set of non-trivial incident edge subgroups in a type M vertex group of $X'$ is a set of cyclic subgroups of a limit group. We claim that these cyclic subgroups are maximal cyclic. By Remark \ref{rmk:limgpmaxcyclic} this is true in the \ice group $G$, and so it follows from the definition of the induced splitting that edge subgroups are also maximal cyclic in type M vertex groups of $X'$. Since $H$ does not contain $\mathbb{Z}^2$, they are therefore maximal abelian, and so malnormality of individual edge subgroups follows from property \ref{item:csaproper} of Theorem \ref{thm:limgpprops}. 
	
	It remains to show that if $a$ and $b$ generate two distinct edge subgroups of a type M vertex group $H_v =\pi_1\left(X'_v\right)$ of $X'$, then conjugates of $\left<a\right> \subset H_v$ and conjugates of $\left<b\right> \subset H_v$ have trivial intersection in $H_v$. Since the \ice group $G$ has only one edge group (which we may take without loss of generality to be $\left<a\right>$) it follows from the definition of the induced splitting that $a$ and $b$ are conjugate as elements of $G$, say $g^{-1}ag=b$, and that the conjugating element $g$ is contained in the vertex group $G_v$ ($=\pi_1(M)$) of $G$, but not in the subgroup $H$. If $a$ and $b$ are also conjugate in $H$, say $h^{-1}ah=b$, then $a$ commutes with $gh^{-1}$ in $G$. Since (by definition of a centralizer extension) $C_{G_v} \left(a\right)=\left<a\right>$, it follows that $gh^{-1} \in \left<a\right>$; in particular, $g \in H$, which is a contradiction. Therefore, $a$ and $b$ cannot be conjugate in $H$.
	
	Now, suppose $t \in H_v$ satisfies $t^{-1}a^p t= b^q$. Since $H_v$ does not contain $\mathbb{Z}^2$, property \ref{item:csa} implies that the two cyclic subgroups generated by $t^{-1}a t$ and $b$ must coincide; but since the edge groups are maximal cyclic, neither $t^{-1}a t$ nor $b$ is a proper power. Thus $t^{-1}a t=b$, which is a contradiction according to the previous paragraph. It follows that conjugates of $\left<a\right> \subset H_v$ and conjugates of $\left<b\right> \subset H_v$ have trivial intersection as required.
	
	It follows that $X'$ satisfies all the conditions of Theorem \ref{thm:main}, and so we can apply it to find the complex $\bar{X}$ as required.
\end{proof}

In summary, we have proved the following theorem:

\begin{theorem}\label{thm:equivalence}
	Let $G$ be a limit group. Then the following are equivalent: \begin{enumerate}
		\item $G$ is hyperbolic. \label{item:hyp}
		\item $G$ is \cat{-1}. \label{item:c-1}
		\item $G$ has \cat{-1} dimension 2. \label{item:c-1dim}
		\item $G$ does not contain $\mathbb{Z}^2$. \label{item:noz2}
	\end{enumerate}
\end{theorem}

\begin{proof}
	\ref{item:c-1dim} $\implies$ \ref{item:c-1} follows from our definition of \cat{-1} dimension (Remark \ref{rmk:cat-1dimcpct}). \ref{item:hyp} $\implies$ \ref{item:noz2} and \ref{item:c-1} $\implies$ \ref{item:hyp} are true for any group. \ref{item:noz2} $\implies$ \ref{item:c-1dim} is precisely Theorem \ref{thm:limnc}.
\end{proof}

Note in particular that this provides an alternative proof of the fact due to Sela \cite[Corollary 4.4]{sela2001}] that a limit group $G$ is hyperbolic if and only if every abelian subgroup is cyclic. Sela's proof also uses a combination theorem, namely that of Bestvina and Feighn \cite{bestvinafeighn92}.

\section{Further applications of the gluing theorem}\label{sec:apps}

There are several contexts in which cyclic splittings of groups are of interest, and our gluing theorem therefore has the potential to shift the question of whether such groups are \mbox{\cat{-1}} to their vertex groups under a cyclic splitting. With this in mind, we can give two more consequences of Theorem \ref{thm:main}. The first consequence concerns JSJ decompositions, for which we will need to recall some technical background. The second consequence concerns graphs of free groups with cyclic edge groups, and this will follow quickly from the JSJ material.  

\subsection{JSJ decompositions of torsion-free hyperbolic groups}

\sloppy JSJ decompositions were originally invented to study toroidal decompositions of 3-manifolds \cite{jaco1979, johannson1979}, and can be thought of as the second stage in the decomposition of a 3-manifold---the stage after cutting along essential spheres. Analogous notions for groups have been studied by Bowditch \cite{bowditch98} (in the case of hyperbolic groups) and Rips--Sela \cite{rips1997} (in the case of general finitely presented groups), among many other generalisations. We will give only the details essential for our argument, and for these we follow \cite{bowditch98}.

Dunwoody's Accessibility Theorem \cite{dunwoody85} shows that any hyperbolic group can be decomposed as a graph of groups whose vertex groups are either finite or one-ended, and whose edge groups are finite. In the torsion-free case, this reduces to a decomposition with trivial edge groups and one-ended, hyperbolic vertex groups (this is also called the Grushko decomposition). Analogously to the 3-manifold setting, the JSJ decomposition then describes how to decompose these components further, and in the torsion-free case, this is a decomposition along infinite cyclic subgroups. 

The following statement is a special case of \cite[Theorem 0.1]{bowditch98}:

\begin{theorem}[JSJ decomposition for torsion-free hyperbolic groups]\label{thm:jsj}
	Let $\Gamma$ be a torsion-free, one-ended hyperbolic group. Then $\Gamma$ is the fundamental group of a well-defined, finite, canonical graph of groups with infinite cyclic edge groups and vertex groups of three types: \begin{description}
			\item[type S] fundamental groups of non-elementary surfaces, whose incident edge groups correspond precisely to the subgroups generated by the boundary components; %changed
			\item[type N] infinite cyclic groups, and
			\item[type M] non-elementary hyperbolic groups not of type S or type N.
		\end{description}
	These three types are mutually exclusive, and no two of the same type are adjacent.
\end{theorem}

In the above statement, a cyclic splitting of $\Gamma$ is called \emph{canonical} if it has a common refinement with any other cyclic splitting of $\Gamma$, where a \emph{refinement} of a splitting is obtained by taking further splittings of vertex groups so that the images of attaching maps in the original splitting are still contained in vertex groups. In this sense, a canonical splitting contains information about any cyclic splitting of the group. The splitting in the JSJ decomposition is well-defined because it is the deepest canonical splitting possible; it has no further canonical refinement.

\begin{remark}
	In Bowditch's original statement, it is taken as a condition that $\Gamma$ is not a cocompact Fuchsian group, which, in the torsion-free case, is the same thing as a closed hyperbolic surface group. We do not need to rule out this case, but note that its JSJ decomposition consists of just one vertex group, of type S, with no incident edge groups. Indeed, a closed surface group cannot have a non-trivial canonical splitting---each cyclic splitting corresponds to a simple closed curve on the surface, and given any such curve, we can always choose another simple closed curve which cannot be homotoped disjoint from it. Then the two splittings defined by these two curves can never admit a common refinement.
\end{remark}

We would like to apply our gluing theorem to the JSJ decomposition of a hyperbolic group. For our theorem to apply, we will need to assume (as before) that the type M vertex spaces have \cat{-1} dimension 2; however, no further assumptions are required due to the following fact:

\begin{lemma}\label{lemma:jsjmalnormal}
	For each type M or type S vertex group in the JSJ decomposition of a hyperbolic group $\Gamma$, the images of incident edge groups form a malnormal family of subgroups.
\end{lemma}

\begin{proof}
	In the type S case, this follows from Theorem \ref{thm:jsj}---in particular, the subgroups concerned are the subgroups generated by the boundary components, and these always form a malnormal family. In the type M case, it follows from the proof of the fact that the action of $\Gamma$ on the Bass--Serre tree corresponding to the JSJ decomposition is 2-acylindrical (see \cite{guirardellevitt11}).
\end{proof}

The next proposition then follows quickly from Theorem \ref{thm:main} and Lemma \ref{lemma:jsjmalnormal}. 

% Can I add compactness into this proposition?

\begin{proposition}\label{prop:jsjnc}
	Let $\Gamma$ be a torsion-free, one-ended hyperbolic group. Suppose all the type M vertex groups in the JSJ decomposition of $\,\Gamma$ are fundamental groups of (compact) negatively curved 2-complexes. Then $\Gamma$ is also the fundamental group of a (compact) negatively curved 2-complex.
\end{proposition}

\begin{proof}
	We construct a graph of spaces corresponding to the JSJ decomposition of $\Gamma$. By definition of the JSJ decomposition, we may choose a circle for each edge space. By assumption, we can choose (compact) negatively curved 2-complexes for the type M vertex spaces, and we can clearly also choose (compact) negatively curved 2-complexes for the type S vertex spaces. We choose a circle for each type N vertex space. If a type M and type S vertex space are adjacent, we insert an additional type N vertex space in between, so that each type M or S vertex space is then only adjacent to type N vertex spaces. Thus the first three conditions of Theorem \ref{thm:main} hold (where type S vertex spaces are included in the type M vertex spaces of Theorem \ref{thm:main}).  Remark \ref{rmk:thmconds} implies that we may assume that the images of attaching maps are then simplicial closed geodesics, hence condition \ref{item:li} also holds, and Lemma \ref{lemma:jsjmalnormal} gives condition \ref{item:malfam}. Hence Theorem \ref{thm:main} applies, and the result follows.
\end{proof}

In the JSJ decomposition, we do not necessarily know any more about the type M vertex groups of $\Gamma$ than we know about $\Gamma$ itself. In particular, they may themselves have a non-trivial JSJ decomposition, or even a free decomposition. However, we may appeal to the following theorem (\cite{loudertouikan13}, see also \cite{delzantpotyagailo01}):

\begin{theorem}[Strong Accessibility Theorem]
	Let $\Gamma$ be a torsion-free hyperbolic group. Consider the hierarchy obtained by taking either the free (if freely decomposable) or JSJ decomposition of $\,\Gamma$, and then taking a free or JSJ decomposition of the resulting vertex groups, and so on. Then this hierarchy is finite.
\end{theorem}

\begin{definition}
	A torsion-free hyperbolic group is called \emph{rigid} if it does not have a non-trivial free or cyclic splitting.
\end{definition}

The Strong Accessibility Theorem says that if we continue decomposing vertex spaces using free products or JSJ decompositions, we must eventually terminate at a decomposition whose vertex groups are rigid (note that vertex groups are always hyperbolic---this is clear for free decompositions, and in the JSJ case is given by Theorem \ref{thm:jsj}). In the context of our gluing theorem, it implies the following %(note that by Remark \ref{rmk:defns} we can assume the subgroups are also 2-dimensional):

% %CHECK THIS, CAT(-1) DIMENSION MAY NOT EQUAL GEOMETRIC DIMENSION A PRIORI

\begin{proposition}\label{prop:jsjrigidsubgp}
	A hyperbolic group $\Gamma$ has \cat{-1} dimension 2 if each rigid subgroup of $\Gamma$ has \cat{-1} dimension 2.
\end{proposition}

\subsection{Graphs of free groups with cyclic edge groups}

In the above subsection, we saw how the gluing theorem can be applied to JSJ decompositions of hyperbolic groups. However, we could only conclude that a group was \cat{-1} if the vertex groups in the JSJ decomposition were \cat{-1} (of dimension 2), which is a strong requirement. Here we describe a context where this requirement is met. 

In \cite{hsuwise10}, Hsu and Wise show that a group $G$ which splits as a finite graph of finitely generated free groups with cyclic edge groups is \cat{0} if and only if it contains no ``non-Euclidean Baumslag--Solitar subgroups'' (subgroups of the form $\left< a, t \; \mid \; t^{-1}a^pt=a^q\right>$ for \mbox{$q\neq \pm p$}). Indeed, under these conditions, they show that the group acts geometrically on a \cat{0} cube complex---a property which has countless interesting consequences (see \cite{wisenotes, wisemanu} for details and further references). However, their methods give little control over the \cat{0} dimension. If $G$ is also hyperbolic (so that it contains no Baumslag Solitar subgroups at all), then we may improve their result in two ways: firstly, showing $G$ is in fact \cat{-1}, and secondly, showing that the \cat{-1} dimension (and hence the \cat{0} dimension) is equal to 2.

\begin{theorem}\label{thm:gofg}
	Let $G$ be a hyperbolic group which splits as a finite graph of finitely generated free groups with cyclic edge groups. Then $G$ has \cat{-1} dimension 2.
\end{theorem}

To prove this, it is tempting to try to apply Theorem \ref{thm:main} directly to the graph of spaces corresponding to the given graph of free groups. However, it may not be possible to ensure that incident edge groups form malnormal families in vertex groups. We circumvent this difficulty by appealing to the JSJ machinery of the previous subsection.

\begin{proof}[Proof of Theorem \ref{thm:gofg}]
%Consider the JSJ decomposition of the group $G$, noting that this is not necessarily the same as the original graph of groups decomposition of $G$. We would like to apply Proposition \ref{prop:jsjrigidsubgp}, and so we must check that each type M vertex group is the fundamental group of a negatively curved complex. So, let $H \subset G$ be a type $M$ vertex group. Consider the splitting of $H$ induced by the original graph of groups splitting of $G$
We would like to apply Proposition \ref{prop:jsjrigidsubgp}, and so we need to check that each rigid subgroup of $G$ has \cat{-1} dimension 2. So, let $H$ be a rigid subgroup of $G$, and consider the splitting of $H$ induced by the decomposition of $G$ as a graph of free groups with cyclic edge groups. Since $H$ is rigid, this induced splitting must consist of a single vertex group, so $H$ is a subgroup of a vertex group of the original splitting. Hence $H$ is free (indeed, since $H$ cannot split freely, it is trivial), and so certainly has \cat{-1} dimension 2, as required.
\end{proof}

\section{Remarks and acknowledgements}

\begin{remark}\label{rmk:malfam}
	We expect that the malnormal family assumption we make in our main theorem can be relaxed---the primary reason to impose it is that it is a common assumption in the literature (see section \ref{ss:malfam}). Indeed, it is used only to show that it is possible to make the corresponding set of closed geodesics transverse (Lemma \ref{lemma:transverse}), which in turn is used only to ensure that there is a safe place to attach the corner of a hyperbolic annulus (see the proof of Theorem \ref{thm:main}). It is easy to design a hyperbolic annulus with \emph{reflex} angle of, say, $\pi+\delta_1$ in the geodesic side, and an angle in the other side of $\pi-\delta_2$, provided $\delta_2>\delta_1$. Using such annuli for edge space cylinders, we may attach multiple annuli along the same closed geodesic in a vertex space. It may be that such a technique allows us to replace the malnormal family assumption with a weaker assumption, reminiscent of the ``annuli flare condition'' used in \cite{bestvinafeighn92}.
\end{remark}

\begin{remark}\label{rmk:error}
	Remark \ref{rmk:compsimplex} does not directly generalise to $n$-dimensional simplices, as pointed out in the introduction of \cite{charneydavis95}. Indeed, there exist hyperbolic 3-simplices whose 1-skeleton is not the 1-skeleton of any Euclidean 3-simplex (nor any 3-simplex of negative curvature $k$ very close to zero). Moreover, even when comparison simplices \emph{can} be found (which can be ensured by choosing $k'$ sufficiently close to $k$), the dihedral angles may \emph{decrease} when passing to a comparison simplex of lesser negative curvature (i.e. $k<k'<0$). This means that Remark \ref{rmk:angleexcess} may not hold in dimensions $>2$, and consequently our proof of Theorem \ref{thm:main} is valid only in two dimensions. However, it may be possible to prove Theorem \ref{thm:main} for higher dimensions, using a similar construction to that in \cite{charneydavis95}. Here, the appropriate generalisation of ``excess angle'' to dimensions $>2$ is the notion of a complex with \emph{extra large links}. An $M_k$-complex is said to have extra large links if the systole of the link of each vertex is strictly greater than $2\pi$, and moreover the links themselves are complexes with extra large links. This is a strictly stronger version of the link condition, with the useful feature that it is stable under small perturbations of the metric (see \cite{moussongthesis} or \cite{davis08}).
\end{remark}

\begin{remark} We have indicated the applicability of our method to families of hyperbolic groups built hierarchically, namely limit groups, and graphs of free groups with cyclic edge groups (via the JSJ decomposition). There are many other families of hyperbolic groups built in a similarly hierarchical way---for example, hyperbolic special groups (see \cite{wisenotes, wisemanu})---and these may lend themselves to investigation using similar techniques. In future work, we intend to find a more general \cat{-1} gluing theorem, valid for larger classes of graphs of spaces---for example, the case where the edge groups are not cyclic; where the vertex spaces are of dimension $>2$ (as discussed above), or where the groups considered are allowed to have torsion.
\end{remark}

We would like to thank Henry Wilton for many helpful suggestions and comments during the preparation of this work. We would also like to thank the anonymous referee for his or her many detailed comments and suggestions which improved the exposition of this paper.

\bibliographystyle{plain}
\bibliography{../phdbib}

\end{document}

%% file: lambertsvg.pdf_tex
%% Creator: Inkscape inkscape 0.48.2, www.inkscape.org
%% Accompanies image file 'lambertsvg.pdf' (pdf, eps, ps)
\begingroup%
  \makeatletter%
  \providecommand\color[2][]{%
    \errmessage{(Inkscape) Color is used for the text in Inkscape, but the package 'color.sty' is not loaded}%
    \renewcommand\color[2][]{}%
  }%
  \providecommand\transparent[1]{%
    \errmessage{(Inkscape) Transparency is used (non-zero) for the text in Inkscape, but the package 'transparent.sty' is not loaded}%
    \renewcommand\transparent[1]{}%
  }%
  \providecommand\rotatebox[2]{#2}%
  \ifx\svgwidth\undefined%
    \setlength{\unitlength}{156.46159668bp}%
    \ifx\svgscale\undefined%
      \relax%
    \else%
      \setlength{\unitlength}{\unitlength * \real{\svgscale}}%
    \fi%
  \else%
    \setlength{\unitlength}{\svgwidth}%
  \fi%
  \global\let\svgwidth\undefined%
  \global\let\svgscale\undefined%
  \makeatother%
  \begin{picture}(1,1.19444356)%
    \put(0,0){\includegraphics[width=\unitlength]{lambertsvg.pdf}}%
    \put(0.17940938,0.7828447){\color[rgb]{0,0,0}\makebox(0,0)[lb]{\smash{$\beta$}}}%
    \put(0.1080543,0.6527114){\color[rgb]{0,0,0}\makebox(0,0)[lb]{\smash{$\beta'$}}}%
    \put(0.64725925,0.19795115){\color[rgb]{0,0,0}\makebox(0,0)[lb]{\smash{$\gamma$}}}%
    \put(0.52918321,0.13276302){\color[rgb]{0,0,0}\makebox(0,0)[lb]{\smash{$\gamma'$}}}%
    \put(0.75203611,0.52576628){\color[rgb]{0,0,0}\makebox(0,0)[lb]{\smash{$b$}}}%
    \put(0.41877944,0.02392963){\color[rgb]{0,0,0}\makebox(0,0)[lb]{\smash{$a$}}}%
    \put(0.08220025,0.42500648){\color[rgb]{0,0,0}\makebox(0,0)[lb]{\smash{$d$}}}%
    \put(0.37393283,0.94839723){\color[rgb]{0,0,0}\makebox(0,0)[lb]{\smash{$c$}}}%
    \put(0.33414499,0.55520875){\color[rgb]{0,0,0}\makebox(0,0)[lb]{\smash{$y$}}}%
    \put(0.52989787,0.67382156){\color[rgb]{0,0,0}\makebox(0,0)[lb]{\smash{$x$}}}%
    \put(0.76042444,0.98689228){\color[rgb]{0,0,0}\makebox(0,0)[lb]{\smash{$\theta$}}}%
    \put(0.15754353,0.00476854){\color[rgb]{0,0,0}\makebox(0,0)[lb]{\smash{$P$}}}%
    \put(0.90667919,1.16046457){\color[rgb]{0,0,0}\makebox(0,0)[lb]{\smash{$R$}}}%
    \put(-0.00569229,0.85301171){\color[rgb]{0,0,0}\makebox(0,0)[lb]{\smash{$S$}}}%
    \put(0.71720905,0.02225814){\color[rgb]{0,0,0}\makebox(0,0)[lb]{\smash{$Q$}}}%
  \end{picture}%
\endgroup%

%% file: Annulus
%!tikz editor 1.0
%!tikz source begin
\begin{tikzpicture}
	\coordinate (V1) at (-1.5,0);
	\coordinate (V2) at (-2.5,3);
	\coordinate (V3) at (2.5,3);
	\coordinate (V4) at (1.5,0);
	\draw[thin]
	(V1) to[bend right=5] coordinate[pos=0.1] (d1) coordinate[pos=0.9] (d2) (V2) to[bend right=5] coordinate[pos=0.1] (c1) coordinate[pos=0.5] (X) coordinate[pos=0.9] (c2) (V3) to[bend right =5] coordinate[pos=0.1] (b1) coordinate[pos=0.9] (b2) (V4) to[bend right=10] coordinate[pos=0.1] (a1) coordinate[pos=0.5] (Y) coordinate[pos=0.9] (a2) (V1);
	\tkzMarkRightAngle(a2,V1,d1);
	\tkzMarkRightAngle(d2,V2,c1);
	\tkzMarkRightAngle(b2,V4,a1);
	\tkzMarkAngle[draw=red, fill=red!25, size=7mm,label={}](c2,V3,b1);
	\tkzMarkAngle[draw=red, fill=red!25, size=7mm,label={}](d2,V2,c1);
	\draw[thick]
	(V1) to[bend right=5] (V2) to[bend right=5] (V3) to[bend right =5] (V4) to[bend right=10] (V1);
	\draw[thick, dashed] (X) to (Y);
	\draw[thick, ->] (3,1.5) to (4,1.5);
	\draw[thick] (5,0) arc (180:360:0.8cm and 0.3cm) coordinate[pos=0.6] (M) ;
	\draw[thick, dotted] (5,0) arc (180:0:0.8cm and 0.3cm);
	\draw[thick] (5,0) to[bend right=15] (4.5,3);
	\draw[thick] (6.6,0) to[bend left=20] (7.1,3);
	\draw[thick] (4.5,3) arc (180:360:1.3cm and 0.4cm) coordinate[pos=0.6] (N) ;
	\draw[thick] (4.5,3) ..controls (4.6,3.3) and (5.2,3.3) .. (5.5,3.6);
	\draw[thick] (7.1,3) ..controls (6.8,3.4) and (5.8,3.3)  .. (5.5,3.6) ;
	\draw[thick, dashed] (M) to[bend left=5] (N);
	\draw[thick,red, dashed] (5.5,3.6) to[bend left=5] (5.7,0.3);
	\draw[thick,red] (V1) to[bend right=5] (V2);
	\draw[thick,red] (V3) to[bend right=5] (V4);
\end{tikzpicture}
%!tikz source end

%% file: Transverse
%!tikz editor 1.0
%!tikz source begin
\begin{tikzpicture}
	\begin{scope}[scale=1.3]
	\coordinate (V1) at (0,0);
	\coordinate (V2) at (1.5,0.4);
	\coordinate (V3) at (3,0.1);
	\coordinate (V4) at (4.5,1);
	\coordinate (V5) at (4.5,-0.4);
	\draw[thick, red]
		([yshift=-1pt]V1) -- node[black, below=2pt] {$e_{j-1}$} ([yshift=-1pt]V2) -- node[black, below] {$e_j$} ([yshift=-1pt]V3) -- node[black, below] {$e_{j+1}$} (V5);
	\draw[thick, blue]
		([yshift=1pt]V1) -- node[black, above] {$e'_{j-1}$} ([yshift=1pt]V2) -- node[black, above] {$e'_j$} ([yshift=1pt]V3) -- node[black, above left=-5pt] {$e'_{j+1}$} (V4);
	\fill[black] (V1) circle (2pt);
	\fill[black] (V2) circle (2pt) node[below=2pt] {$v_{j-1}$};
	\fill[black] (V3) circle (2pt) node[below=2pt] {$v_{j}$};
	\fill[black] (V4) circle (2pt) node[below=2pt] {$v'_{j+1}$};
	\fill[black] (V5) circle (2pt);
	\end{scope}	
		\draw[thick, ->]
		(6,0.3) -- (7,0.3);
	\begin{scope} [shift={(7.5,0)}, scale=1.3]
	\coordinate (V1) at (0,0);
	\coordinate (V2) at (1.5,0.4);
	\coordinate (V3) at (3,0.1);
	\coordinate (V4) at (4.5,1);
	\coordinate (V5) at (4.5,-0.4);
	\path [fill=blue!15]
	([yshift=-1pt]V2) to ([yshift=1pt]V2) to[bend right=8] (V4) to ([yshift=0pt]V3);
	\draw[thick, red]
		([yshift=-1pt]V1) -- node[black, below=2pt] {$e_{j-1}$} ([yshift=-1pt]V2) -- node[black, below] {$e_j$} ([yshift=-1pt]V3) -- node[black, below] {$e_{j+1}$} (V5);
	\draw[thick, blue, dashed]
		([yshift=0pt]V2) --	([yshift=0pt]V3) -- (V4);
	\draw[thick, blue]
		([yshift=1pt]V1) -- node[black, above] {$e'_{j-1}$} ([yshift=1pt]V2) to[bend right=8] node[black, above] {$e'$} (V4);
	\fill[black] (V1) circle (2pt);
	\fill[black] (V2) circle (2pt) node[below=2pt] {$v_{j-1}$};
	\fill[black] (V3) circle (2pt) node[below=2pt] {$v_{j}$};
	\fill[black] (V4) circle (2pt) node[below=2pt] {$v'_{j+1}$};
	\fill[black] (V5) circle (2pt);
	\end{scope}
\end{tikzpicture}
%!tikz source end